\documentclass[11pt]{article}
\usepackage[T1]{fontenc}
\usepackage{lmodern}
\usepackage[margin=1in]{geometry}
\usepackage{amsmath,amssymb,amsthm}
\usepackage{booktabs,microtype,titlesec}
\usepackage{etoolbox,placeins}
\titleformat{\section}{\large\normalfont}{\thesection.}{0.5em}{}
\newtheoremstyle{narrative}{6pt}{6pt}{\normalfont}{}{\normalfont}{.}{ }{}
\theoremstyle{narrative}
\patchcmd{\abstract}{\bfseries}{\normalfont}{}{}
\usepackage{tikz}
\usetikzlibrary{arrows.meta}
\usepackage[hidelinks]{hyperref}
\hypersetup{
  pdftitle={Edge complexity of weighted graphs: involutory symmetries and NP-hardness},
  pdfauthor={Vishal Gupta and Alex Iosevich},
  pdfsubject={Fourier ratios, weighted graphs, and computational complexity}
}
\newtheorem{theorem}{Theorem}[section]
\newtheorem{proposition}[theorem]{Proposition}
\newtheorem{lemma}[theorem]{Lemma}
\newtheorem{corollary}[theorem]{Corollary}
\newtheorem{definition}[theorem]{Definition}
\newtheorem{example}[theorem]{Example}
\newtheorem{question}[theorem]{Question}
\newtheorem{remark}[theorem]{Remark}

\title{Edge complexity of weighted graphs: involutory symmetries and NP-hardness}

\author{Vishal Gupta\thanks{Department of Mathematics, University of Rochester, Rochester, New York 14627, USA. Email address: \nolinkurl{vishalgupta@rochester.edu}.} \and Alex Iosevich\thanks{Department of Mathematics, University of Rochester, Rochester, New York 14627, USA. Email address: \nolinkurl{alex.iosevich@rochester.edu}. The second listed author was supported in part by NSF grant DMS-2506858.}}

\date{}

\begin{document}
\maketitle

\begin{abstract}
The edge complexity of a weighted graph is the smallest ratio of the Fourier $\ell^1$ and $\ell^2$ norms of its adjacency matrix over all vertex labelings. We study the difficulty of finding this minimum by relating it to a graph symmetry. Adding a universal vertex with sufficiently large incident weight produces an explicit Fourier $\ell^1$ lower bound. We show that equality holds exactly when the source graph has a fixed-point-free involutory automorphism. Two-sided estimates compare the excess above this bound with the squared Frobenius distance to the nearest weighted graph having such a symmetry. For sources of constant weighted degree, these estimates determine the exact leading term as the added weight tends to infinity. A stronger separation for simple source graphs proves that additive $\frac{1}{256N^{\frac{7}{2}}}$ approximation of weighted edge complexity is NP-hard, even on connected graphs of odd order $N$ with at most two distinct positive integer weights, each at most $N^2$. We also prove that recognizing a simple graph with a real Fourier labeling is NP-complete. A seven-vertex example shows that every minimizing labeling can have nonreal Fourier coefficients even when real Fourier labelings exist. An exact rational certificate for this example is included in the appendix.
\end{abstract}



\section{Introduction}
A graph does not come with a preferred ordering of its vertices. Once we choose an ordering, however, its adjacency matrix becomes a function on $\mathbb{Z}_N\times\mathbb{Z}_N$, where $N$ is the number of vertices. We can then ask whether its Fourier transform is concentrated on a small collection of frequencies. Different orderings can give very different answers. The natural question is therefore to find the ordering for which the Fourier transform is least complicated.

One way to make this question precise is to divide the sum of the absolute values of the Fourier coefficients by their Euclidean norm. This is the Fourier ratio. The denominator is unchanged by a relabeling, so minimizing the ratio amounts to minimizing the numerator. Gupta and Iosevich \cite{GI} introduced the resulting graph invariant, called edge complexity, and related it to graph energy and recovery from partial information. That work also treats weighted and directed kernels and perturbation estimates. Further results on equality in the energy bound, graph products, and random graphs appear in \cite{GIIST}. A different approach to edge complexity, based on additive energy of
the labeled edge set, is developed in forthcoming work \cite{additive-energy}. For the broader Fourier-ratio framework, see \cite{Aldaleh,Burstein}.

The question we address is whether this minimum can be found efficiently. For a specified labeling, the Fourier ratio can be evaluated to any prescribed accuracy in polynomial time. The difficulty is the choice of labeling. We prove that approximating the minimum to inverse-polynomial additive accuracy is NP-hard, even for a restricted family of connected graphs with nonnegative integer weights.

In a network, an edge weight may record a capacity, a similarity, or the number of interactions between two vertices; see Newman \cite{Newman}. Here the weights also let us arrange a useful separation of scales. We add one vertex, join it to all the old vertices, and give the new edges a common large weight. We place the added vertex at label zero. The resulting star determines the sign of the real part of every Fourier coefficient. The old graph supplies a smaller contribution, which can change the imaginary parts but cannot change those signs.

The connection with graph symmetry comes from a simple equality condition. For a complex number $z$, the inequality $|z|\ge|\operatorname{Re}z|$ is an equality precisely when $z$ is real. In our construction, the sum of the absolute values of the real parts is independent of the labeling. Equality in the resulting lower bound therefore asks that every Fourier coefficient be real. Fourier conjugation translates this last condition into invariance under the reflection $x\mapsto-x$. Since the new graph has odd order, reflection fixes the added vertex and pairs all the old vertices. This is the symmetry that the construction detects.

Let $H$ be a nonnegative weighted graph of even order $n\ge4$, and let $s$ be the sum of its edge weights. Add a vertex adjacent to every vertex of $H$, assigning weight $M>s$ to each new edge.
The resulting weighted cone is denoted by $W_M(H)$, and its order is $N=n+1$. Theorem~\ref{thm:cone} proves that its minimum Fourier $\ell^1$ norm is at least
$$
T=\frac{4Mn^2-4s}{N}.
$$
Equality holds precisely when $H$ has an automorphism consisting entirely
of transpositions. Such an automorphism is called a fixed-point-free involution. 
The equality statement has a quantitative counterpart. If $B_H$ is the weighted adjacency matrix of $H$, define
$$
a(B_H)=\min_\tau\frac{1}{4}
\|B_H-P_\tau B_HP_\tau\|_2^2,
$$
where the minimum is over all fixed-point-free involutions of the vertex set and $P_\tau$ is the corresponding permutation matrix. Here and below, the matrix $\ell^2$ norm is the Frobenius norm. This quantity is the squared distance to the nearest nonnegative weighted adjacency matrix having the required symmetry. Theorem~\ref{thm:stability} bounds the excess above $T$ from both sides in terms of $a(B_H)$. If all weighted degrees of $H$ are equal, Theorem~\ref{thm:regular} gives the more precise conclusion
$$
\lim_{M\to\infty}M\bigl(\Phi(W_M(H))-T\bigr)
=\frac{N}{4}a(B_H),
$$
where $\Phi$ denotes the minimum Fourier $\ell^1$ norm.

For a simple source graph, $a(B_H)$ has a direct combinatorial meaning: it is the smallest number of edge additions or deletions required to obtain a graph with a fixed-point-free involutory automorphism. If this number is nonzero, Theorem~\ref{thm:simple-gap} proves the uniform separation
$$
\Phi(W_M(H))\ge T+\frac{N}{16M}.
$$
This estimate supplies the gap needed for a complexity reduction. Taking $M=|E(H)|+1$, we obtain NP-hardness of approximating weighted edge complexity to additive error $\frac{1}{256N^{\frac{7}{2}}}$. The constructed graphs are connected, have odd order, and use only the positive weights $1$ and $M$, with $M\le N^2$.

The distinction between the analytic and computational statements is useful. The analytic results hold for nonnegative real weights. The computational results concern rational weights represented in binary, and the reduction uses only integer weights. This distinction matters as unequal real weights can be arbitrarily close, so failure of symmetry need not produce a uniform gap without an additional separation assumption.

For simple graphs, we prove that deciding whether a labeling with a real Fourier matrix exists is NP-complete, even for connected graphs of odd order with a unique universal vertex. The complexity of computing $\operatorname{FR}_{\min}$ for simple graphs remains unresolved by our argument. The large weight ensures the sign condition used in the equality characterization for the cone, and Example~\ref{ex:sign-failure} shows that this condition can fail when the new edges have weight one. Moreover, we give a seven-vertex graph admitting real Fourier labelings for which every minimizing labeling has a nonreal Fourier matrix. Thus recognizing a real Fourier labeling and finding a minimizing labeling are different questions.

We begin with the normalization and elementary bounds, then prove the cone theorem and its quantitative versions. The complexity reduction uses only these estimates and the NP-completeness of detecting a fixed-point-free involution. We conclude by examining what the argument says about simple graphs and by giving an exact certificate for the seven-vertex example.

\section{Weighted edge complexity}

We first fix the normalization and separate the elementary properties of the invariant from the symmetry argument.

\begin{samepage}
\begin{definition}\label{def:weighted}
A nonnegative weighted graph is a pair $G=(V,w)$, where $V$ is a finite set and $w:\binom{V}{2}\to[0,\infty)$. Its weighted adjacency matrix $W=(w_{uv})$ is symmetric and has zero diagonal. Its support graph has edge set
$$
E_w=\{\{u,v\}\in\binom{V}{2}:w_{uv}>0\}.
$$
Let
$$
s_w=\sum_{\{u,v\}\in\binom{V}{2}}w_{uv}
 \text{ and } 
t_w=\sum_{\{u,v\}\in\binom{V}{2}}w_{uv}^2.
$$
\end{definition}
\end{samepage}

Let $N=|V|$. A labeling is a bijection $\pi:\mathbb{Z}_N\to V$. It turns
the weighted adjacency matrix into the function
$$
f_\pi(x,y)=w_{\pi(x)\pi(y)}.
$$
We also denote this labeled matrix by $W_\pi$. The same permutation acts
on the rows and columns, because both indices refer to the same vertex
set. Define
\begin{equation}\label{eq:fourier}
\widehat f_\pi(r,v)=\frac{1}{N}
\sum_{x,y\in\mathbb{Z}_N}f_\pi(x,y)
e^{-\frac{2\pi i(rx+vy)}{N}}.
\end{equation}
The factor $\frac{1}{N}$ makes this transform unitary on the $N^2$
ordered pairs. If
$$
F_N(r,x)=\frac{1}{\sqrt{N}}e^{-\frac{2\pi irx}{N}},
$$
then $\widehat f_\pi=F_NW_\pi F_N^{\mathsf T}=F_NW_\pi F_N$, since the Fourier matrix is symmetric. Both variables in \eqref{eq:fourier} carry the same negative sign. This accounts for the factor $F_N$ on the right; conjugating that factor would give a different transform. 

For a matrix $C$,
$$
\|C\|_1=\sum_{r,v}|C_{rv}|
 \text{ and } 
\|C\|_2=\left(\sum_{r,v}|C_{rv}|^2\right)^{\frac{1}{2}}.
$$
The second is also called the Frobenius norm.
Each undirected edge appears twice in the adjacency matrix, so Parseval's
identity gives
\begin{equation}\label{eq:parseval}
\|\widehat f_\pi\|_2=\|W_\pi\|_2=\|W\|_2=\sqrt{2t_w}.
\end{equation}

\begin{samepage}
\begin{definition}\label{def:fr}
For a nonzero weighted adjacency matrix $W$, set
$$
\Phi(W)=\min_\pi\|F_NW_\pi F_N\|_1.
$$
The edge complexity of the corresponding weighted graph is
\begin{equation}\label{eq:frmin}
\operatorname{FR}_{\min}(W)
=\min_\pi\frac{\|F_NW_\pi F_N\|_1}{\|F_NW_\pi F_N\|_2}
=\frac{\Phi(W)}{\|W\|_2}.
\end{equation}
We also write $\operatorname{FR}_{\min}(G)$. 
\end{definition}
\end{samepage}

For a simple graph, the weights belong to $\{0,1\}$. Thus $f_\pi$ is
exactly the indicator of the ordered edge relation, $t_w=|E(G)|$, and
\eqref{eq:frmin} becomes the edge complexity defined in \cite{GI}. The
weighted definition keeps the same Fourier transform and the same minimization over labelings. To see why the ratio measures concentration, suppose a Fourier matrix has $k$ nonzero entries. Cauchy--Schwarz bounds its Fourier ratio by $\sqrt{k}$, with equality when those entries have equal absolute value. The ratio therefore records both the number of frequencies present and how their magnitudes are distributed.

\begin{samepage}
\begin{proposition}\label{prop:basic}
Let $W$ be a nonzero nonnegative weighted adjacency matrix of order $N$.
For every $c>0$,
$$
\operatorname{FR}_{\min}(cW)=\operatorname{FR}_{\min}(W).
$$
For a fixed labeling, an affine change of labels $x\mapsto ax+b$, with
$a$ a unit in $\mathbb{Z}_N$ and $b\in\mathbb{Z}_N$, preserves the Fourier
ratio. If
$$
\mathcal E(W)=\sum_{j=1}^N|\lambda_j(W)|
$$
is the matrix energy, then
\begin{equation}\label{eq:energy}
\sqrt{2}\le\frac{\mathcal E(W)}{\|W\|_2}
\le\operatorname{FR}_{\min}(W)\le N.
\end{equation}
Equality in the middle inequality holds whenever $W$ has a circulant
labeling, meaning that its entries depend only on the difference of the
two labels.
\end{proposition}
\end{samepage}
\begin{proof}[\normalfont Proof]
Scaling multiplies both Fourier norms by the same number. For the affine
claim, give the vertex with old label $x$ the new label $ax+b$. Substitution
in \eqref{eq:fourier} gives
$$
\widehat f'(r,v)
=e^{-\frac{2\pi ib(r+v)}{N}}\widehat f(ar,av).
$$
The factor has absolute value one, and multiplication by $a$ permutes the
frequencies. Both norms are therefore unchanged.

To prove the energy bound, recall that the nuclear norm of a matrix is the
sum of its singular values. Denote it by $\|\cdot\|_{S_1}$. Multiplication
by unitary matrices preserves singular values. Moreover, if $E_{rv}$ is a
matrix unit, then $\|E_{rv}\|_{S_1}=1$, so the triangle inequality gives
$$
\|C\|_{S_1}
\le\sum_{r,v}|C_{rv}|\|E_{rv}\|_{S_1}=\|C\|_1.
$$
Since $W$ is real symmetric, its singular values are the absolute values
of its eigenvalues. For every labeling,
$$
\mathcal E(W)=\|F_NW_\pi F_N\|_{S_1}
\le\|F_NW_\pi F_N\|_1.
$$
Taking the minimum and dividing by $\|W\|_2$ proves the middle inequality.

The zero diagonal gives $\operatorname{tr}W=0$. The positive eigenvalues
and the absolute values of the negative eigenvalues consequently have the
same sum, say $p>0$. Squaring a sum of nonnegative numbers can only
increase the sum of their squares. Hence
$$
\|W\|_2^2=\sum_j\lambda_j(W)^2\le2p^2,
$$
while $\mathcal E(W)=2p$. This proves the first inequality. The last one
is Cauchy--Schwarz applied to the $N^2$ Fourier entries.

Finally, suppose $W$ is circulant in a chosen labeling. Its character
vectors are eigenvectors, so $F_NWF_N^*$ is diagonal. Let $R_N$ be the
permutation matrix for $x\mapsto-x$. The identities $F_N^*=F_NR_N$ and
$R_N^2=I$ show that $F_NWF_N$ is obtained from that diagonal matrix by
permuting columns. Its entrywise norm is therefore $\mathcal E(W)$, which
attains the lower bound.
\end{proof}

The minimum also varies continuously with the weights. It is useful to
state this before imposing any arithmetic assumptions on the input.

\begin{samepage}
\begin{proposition}\label{prop:continuity}
Extend $\Phi$ to all real symmetric matrices by the same formula. Then
\begin{equation}\label{eq:phi-lipschitz}
|\Phi(U)-\Phi(V)|\le N\|U-V\|_2.
\end{equation}
For nonzero $U,V$, extending the Fourier-ratio definition in the same way,
\begin{align}
|\operatorname{FR}_{\min}(U)-\operatorname{FR}_{\min}(V)|
&\le N\left\|\frac{U}{\|U\|_2}-\frac{V}{\|V\|_2}\right\|_2
\label{eq:normalized-lipschitz}\\
&\le\frac{2N\|U-V\|_2}{\min\{\|U\|_2,\|V\|_2\}}.
\nonumber
\end{align}
\end{proposition}
\end{samepage}
\begin{proof}[\normalfont Proof]
For any fixed labeling, the triangle inequality, Cauchy--Schwarz, and
Parseval give
$$
\left|\|F_NU_\pi F_N\|_1-\|F_NV_\pi F_N\|_1\right|
\le\|F_N(U_\pi-V_\pi)F_N\|_1
\le N\|U-V\|_2.
$$
Choose a minimizing labeling for $V$ to obtain one direction of
\eqref{eq:phi-lipschitz}, and interchange $U,V$ for the other direction.
Apply that estimate to the normalized matrices for the first inequality
in \eqref{eq:normalized-lipschitz}. If $a=\|U\|_2$ and $b=\|V\|_2$, then
$$
\left\|\frac{U}{a}-\frac{V}{b}\right\|_2
\le\frac{\|U-V\|_2}{a}+\frac{|a-b|}{a}
\le\frac{2\|U-V\|_2}{a}.
$$
Here the last step is the reverse triangle inequality. Replacing $a$ in
the denominator by $\min\{a,b\}$ proves the displayed assertion.
\end{proof}

\section{Reflection symmetry and a weighted cone}

We now ask what it means for every Fourier coefficient of a real adjacency
matrix to be real. The answer comes directly from conjugating the Fourier
transform. 

An involution is a permutation $\tau$ satisfying $\tau^2=\operatorname{id}$.
It is fixed-point-free if $\tau(v)\ne v$ for every vertex. An automorphism
of a weighted graph is a permutation preserving all weights, including
zero weights.

\begin{samepage}
\begin{lemma}\label{lem:reflection}
Let $A$ be a real matrix indexed by $\mathbb{Z}_N$, and let $R=R_N$ represent
$x\mapsto-x$. Then
\begin{equation}\label{eq:conjugation}
\overline{F_NAF_N}=F_N(RAR)F_N.
\end{equation}
In particular,
\begin{equation}\label{eq:real-iff}
F_NAF_N\text{ is real if and only if }A=RAR,
\end{equation}
and
\begin{equation}\label{eq:imag-identity}
\|\operatorname{Im}(F_NAF_N)\|_2^2
=\frac{1}{4}\|A-RAR\|_2^2.
\end{equation}
\end{lemma}
\end{samepage}
\begin{proof}[\normalfont Proof]
The Fourier matrix satisfies $\overline{F_N}=F_NR=RF_N$. Since $A$ is
real, this proves \eqref{eq:conjugation}. The Fourier matrix is invertible,
so equality of the transform with its conjugate is equivalent to $A=RAR$.
Finally,
$$
\operatorname{Im}(F_NAF_N)=\frac{1}{2i}F_N(A-RAR)F_N.
$$
Taking the squared Frobenius norm and using unitarity proves
\eqref{eq:imag-identity}.
\end{proof}

We say that a labeling is a real Fourier labeling if every entry of its Fourier matrix is real. Reflection fixes the solutions of $2x=0$ in $\mathbb{Z}_N$, of which there
are $\gcd(2,N)$. For $N\ge3$, a graph therefore has a real Fourier
labeling precisely when it has an involutory automorphism with one fixed
vertex if $N$ is odd, or two fixed vertices if $N$ is even. For the converse
direction, assign the fixed vertices the fixed labels and assign each
transposed pair a pair of opposite labels. This explains why an odd-order
cone over an even-order source is well suited to detecting a
fixed-point-free involution of the source.

Let $H$ be a nonnegative weighted graph of even order $n\ge4$, with
adjacency matrix $B_H$. Let
$$
s=\sum_{u<v}(B_H)_{uv},\quad 
t=\sum_{u<v}(B_H)_{uv}^2,\quad 
  N=n+1.
$$
For $M>s$, form $W_M(H)$ by adjoining a vertex $c$ and giving every edge
from $c$ to $H$ weight $M$. We use the same notation for the weighted graph
and its adjacency matrix when no confusion can arise. Its Frobenius norm is
\begin{equation}\label{eq:cone-norm}
Q=\sqrt{2nM^2+2t}.
\end{equation}

After placing the new vertex at label zero, the star has a particularly simple Fourier transform: it is positive on the two coordinate axes and negative away from them. The condition $M>s$ ensures that adding the source cannot reverse any of these signs. We can therefore add the signed real parts without first knowing the labeling. Lemma~\ref{lem:reflection} will then identify the equality case.

\begin{samepage}
\begin{theorem}\label{thm:cone}
With the notation above, let
\begin{equation}\label{eq:T}
T=\frac{4Mn^2-4s}{N}.
\end{equation}
Then
\begin{equation}\label{eq:cone-bound}
\Phi(W_M(H))\ge T,
\operatorname{FR}_{\min}(W_M(H))\ge\frac{T}{Q}.
\end{equation}
Equality holds if and only if $H$ has a fixed-point-free involutory
automorphism preserving its weights.
\end{theorem}
\end{samepage}
\begin{proof}[\normalfont Proof]
Translation invariance allows us to give $c$ label $0$. We call such a
labeling centered. Write the labeled cone as $W=MS+B$, where $S$ is the
adjacency matrix of the star centered at $0$, and $B$ is the labeled source
extended by a zero row and column at $0$. Put $D=F_NBF_N$ and $C=F_NWF_N$.

The star is simple enough to transform explicitly. Its kernel is
$$
S(x,y)=1_{\{x=0\}}+1_{\{y=0\}}-2 1_{\{x=y=0\}}.
$$
Summing the characters in each term gives
\begin{equation}\label{eq:star-transform}
(F_NSF_N)_{rv}
=1_{\{r=0\}}+1_{\{v=0\}}-\frac{2}{N}
=\begin{cases}
\frac{2n}{N},&r=v=0,\\
\frac{n-1}{N},&\text{exactly one index is zero},\\
-\frac{2}{N},&r,v\ne0.
\end{cases}
\end{equation}
Because the total weight of the ordered source edges is $2s$, $|D_{rv}|\le\frac{2s}{N}$. Define
$$
\sigma_{rv}=\begin{cases}
1,&r=0\text{ or }v=0,\\
-1,&r,v\ne0.
\end{cases}
$$
The star contribution dominates the source contribution. More precisely,
\begin{equation}\label{eq:sign-margin}
b_{rv}:=\sigma_{rv}\operatorname{Re}C_{rv}
\ge\frac{2(M-s)}{N}>0.
\end{equation}
Away from the axes this follows from the last line of \eqref{eq:star-transform}. On an axis, away from the origin, it follows from $(n-1)M-2s\ge2(M-s)$, using $n\ge4$. At the origin the coefficient is positive, and the same lower bound holds.

It follows entry by entry that
\begin{equation}\label{eq:signed-bound}
\|C\|_1\ge\sum_{r,v}\sigma_{rv}\operatorname{Re}C_{rv},
\end{equation}
with equality if and only if every entry of $C$ is real. To calculate the right side, first sum the three types of star entries. There is one origin,
there are $2n$ other axis entries, and there are $n^2$ entries away from
the axes. Thus, their signed sum is
$$
\frac{2n+2n(n-1)+2n^2}{N}=\frac{4n^2}{N}.
$$

The source contribution is also independent of its labeling. Character orthogonality and the zero row and column of $B$ give
\begin{equation}\label{eq:sum-identities}
\sum_{r,v}D_{rv}=0,
 \sum_vD_{0v}=\sum_rD_{r0}=0,
  D_{00}=\frac{2s}{N}.
\end{equation}
For example, summing $D_{0v}$ over $v$ leaves $\sum_xB_{x0}=0$, and summing over both indices leaves $NB_{00}=0$. Since
$$
\sigma_{rv}=-1+2 1_{\{r=0\}}+2 1_{\{v=0\}}
-2 1_{\{r=v=0\}},
$$
we obtain $\sum_{r,v}\sigma_{rv}D_{rv}=-\frac{4s}{N}$. Substitution in \eqref{eq:signed-bound} proves $\|C\|_1\ge T$.

It remains to identify equality. By Lemma~\ref{lem:reflection}, the Fourier matrix is real if and only if $W=RWR$. The star is already reflection-invariant, so this is equivalent to $B=RBR$. Since $N$ is odd, reflection fixes only $0$ and pairs all source labels. Its restriction is therefore a fixed-point-free involutory automorphism of $H$. Conversely, given such an automorphism, assign its transposed vertex pairs the label pairs $\{j,-j\}$ for $1\le j\le\frac{n}{2}$. Then $B=RBR$, the Fourier matrix is real, and equality holds. There are finitely many labelings, so the same characterization applies to the minimum. Division by \eqref{eq:cone-norm} proves the Fourier-ratio statement.
\end{proof}

\begin{figure}[htbp]
\centering
\begin{tikzpicture}[scale=0.87,
 vertex/.style={circle,draw,fill=white,minimum size=7mm,inner sep=1pt},
 source/.style={line width=0.9pt},
 cone/.style={draw=blue!60!black,line width=0.7pt}]
\node at (1,2.1) {Source $H$};
\node at (8,2.1) {Weighted cone $W_M(H)$};
\coordinate (a) at (0,1);
\coordinate (b) at (0,-1);
\coordinate (d) at (2,1);
\coordinate (e) at (2,-1);
\draw[source] (a)--node[left] {$1$}(b);
\draw[source] (d)--node[right] {$1$}(e);
\node[vertex] at (a) {$1$};
\node[vertex] at (b) {$2$};
\node[vertex] at (d) {$-2$};
\node[vertex] at (e) {$-1$};
\draw[-{Stealth},line width=0.7pt] (3.1,0)--(4.7,0);
\node at (3.9,0.8) {add a vertex};
\node at (3.9,-0.8) {weight $M$};
\coordinate (f) at (6.5,1);
\coordinate (g) at (6.5,-1);
\coordinate (h) at (9.5,1);
\coordinate (j) at (9.5,-1);
\coordinate (c) at (8,0);
\draw[cone] (c)--(f);
\node at (7.3,1.1) {$M$};
\draw[cone] (c)--(g) (c)--(h) (c)--(j);
\draw[source] (f)--node[left] {$1$}(g);
\draw[source] (h)--node[right] {$1$}(j);
\node[vertex] at (f) {$1$};
\node[vertex] at (g) {$2$};
\node[vertex] at (h) {$-2$};
\node[vertex] at (j) {$-1$};
\node[vertex] at (c) {$0$};
\node at (1,-1.9) {$s=2$};
\node at (8,-1.9) {$M>2$};
\end{tikzpicture}
\caption{The cone construction for the illustrative source $H=2K_2$.
Vertex labels are inside circles and edge weights are outside. All four
blue edges have weight $M$. Reflection pairs $1$ with $-1$ and $2$ with
$-2$, fixes the cone vertex, and preserves the source. The large weight
prescribes the Fourier signs; the source symmetry then gives equality.}
\label{fig:cone}
\end{figure}
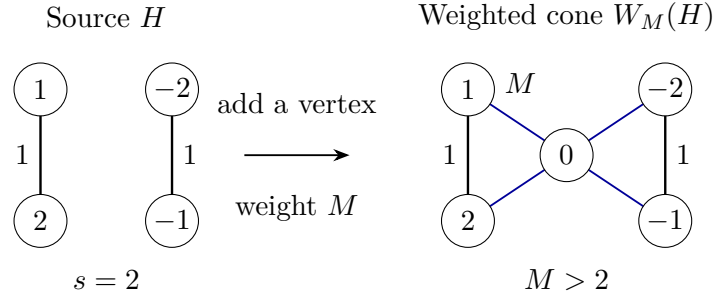
Figure~\ref{fig:cone} shows the construction when the source consists of two disjoint edges. No integrality entered the proof. The role of $M>s$ was to fix the signs of the real parts. Once that has happened, the elementary inequality $|z|\ge\sigma\operatorname{Re}z$ has exactly the equality condition we need.

\section{Distance from symmetry}

The cone theorem detects whether a symmetry exists. We next ask how the Fourier norm changes when the symmetry fails by a small amount. The appropriate measure of failure is already present in \eqref{eq:imag-identity}.

Let $\mathcal I_n$ be the set of fixed-point-free involutions of the source vertex set. For $\tau\in\mathcal I_n$, write $P_\tau$ for its permutation matrix and set
\begin{equation}\label{eq:def-a}
q_\tau(B_H)=\frac{1}{4}\|B_H-P_\tau B_HP_\tau\|_2^2,
  a(B_H)=\min_{\tau\in\mathcal I_n}q_\tau(B_H).
\end{equation}
The matrix $P_\tau$ equals its transpose and inverse. In particular, $a(B_H)=0$ precisely when the required automorphism exists.

\begin{samepage}
\begin{proposition}\label{prop:projection}
For fixed $\tau\in\mathcal I_n$, the matrix
$$
B_\tau^+=\frac{1}{2}(B_H+P_\tau B_HP_\tau)
$$
is the unique closest $\tau$-invariant nonnegative symmetric zero-diagonal
matrix to $B_H$ in Frobenius norm, and
\begin{equation}\label{eq:projection}
\|B_H-B_\tau^+\|_2^2=q_\tau(B_H).
\end{equation}
For a simple source graph,
\begin{equation}\label{eq:edit-count}
q_\tau(B_H)=\frac{1}{2}|E(H)\mathbin\triangle\tau(E(H))|
\end{equation}
is the minimum number of edge additions or deletions needed to make
$\tau$ an automorphism.
\end{proposition}
\end{samepage}
\begin{proof}[\normalfont Proof]
Conjugation by $P_\tau$ is an orthogonal involution on the space of real
matrices with the Frobenius inner product. Averaging a matrix with its
image is the orthogonal projection onto the fixed subspace. In particular,
for every invariant matrix $X$,
$$
\|B_H-X\|_2^2
=\|B_H-B_\tau^+\|_2^2+\|B_\tau^+-X\|_2^2.
$$
The average preserves nonnegativity, symmetry, and zero diagonal. This
proves the minimizing assertion, uniqueness, and \eqref{eq:projection}.

To interpret the expression combinatorially, let $\tau$ act on unordered
vertex pairs. Each orbit has one or two elements. A one-element orbit
contributes zero to $q_\tau$. A two-element orbit with weights $a,b$
contributes $(a-b)^2$: its four ordered matrix entries each contribute
$(a-b)^2$ before division by four. For a simple graph, this contribution
is one exactly when one pair is an edge and the other is not. Adding the
missing edge or deleting the existing edge repairs that orbit. Different
orbits can be repaired independently, proving \eqref{eq:edit-count}.
\end{proof}

The distance statement and the edit statement have slightly different meanings. The closest weighted matrix
$B_\tau^+$ may have entries $\frac{1}{2}$ even when $B_H$ is simple. If the
approximating matrix is required to remain simple, its minimum squared
Frobenius distance is $2q_\tau(B_H)$, because each undirected edge edit
changes two matrix entries. The minimum number of edits is $q_\tau(B_H)$. Thus the weighted squared distance and the unweighted edit count have the same numerical value, although their minimizing matrices need not coincide.

To connect this distance with the Fourier norm, consider one coefficient whose real part already has the prescribed sign. Its absolute value exceeds the signed real part by a quantity quadratic in the imaginary part. The next proof writes this observation as an exact identity and then sums it over all frequencies.

\begin{samepage}
\begin{theorem}\label{thm:stability}
Under the assumptions of Theorem~\ref{thm:cone},
\begin{equation}\label{eq:stability}
\frac{a(B_H)}{4M}
\le\Phi(W_M(H))-T
\le\frac{Na(B_H)}{4(M-s)}.
\end{equation}
The corresponding Fourier-ratio estimates follow by division by $Q$.
If $M\ge2s$ and $M>0$, the upper bound is at most $\frac{Na(B_H)}{2M}$.
\end{theorem}
\end{samepage}
\begin{proof}[\normalfont Proof]
Fix a centered labeling and use $C,D,b_{rv}$ from the cone proof. For an
entry $z=C_{rv}$,
\begin{equation}\label{eq:excess-identity}
|z|-b_{rv}
=\frac{(\operatorname{Im}z)^2}{|z|+b_{rv}}.
\end{equation}
Indeed, multiplying by the denominator gives
$|z|^2-b_{rv}^2=(\operatorname{Im}z)^2$. The denominator is positive by
\eqref{eq:sign-margin}. The total ordered edge weight of the cone is $2(nM+s)$, so
$$|z|\le\frac{2(nM+s)}{N}<2M.$$
It follows that $|z|+b_{rv}<4M$. In the other direction, \eqref{eq:sign-margin} gives
$$|z|+b_{rv}\ge2b_{rv}\ge\frac{4(M-s)}{N}.$$
The star transform is real, so $\operatorname{Im}C=\operatorname{Im}D$. The reflection identity therefore shows that the sum of the squared imaginary parts is $q_\tau(B_H)$, where $\tau$ is induced by this labeling. Since $\sum_{r,v}b_{rv}=T$, summing \eqref{eq:excess-identity} gives
\begin{equation}\label{eq:label-stability}
\frac{q_\tau(B_H)}{4M}
\le\|C\|_1-T
\le\frac{Nq_\tau(B_H)}{4(M-s)}.
\end{equation}
Every centered labeling has $q_\tau\ge a(B_H)$, proving the lower bound
for the minimum. For the upper bound, choose a pairing attaining $a(B_H)$
and label its pairs by opposite residues. This proves the claim.
\end{proof}

\begin{samepage}
\begin{corollary}\label{cor:quantized}
Suppose distinct values among the source weights, including zero, differ
by at least $h>0$. If $H$ has no fixed-point-free involutory automorphism,
then
$$
\Phi(W_M(H))\ge T+\frac{h^2}{4M},
\operatorname{FR}_{\min}(W_M(H))\ge\frac{T}{Q}+\frac{h^2}{4MQ}.
$$
In particular, this applies when all weights are nonnegative integer
multiples of $h$.
\end{corollary}
\end{samepage}
\begin{proof}[\normalfont Proof]
Every pairing has an edge orbit whose two weights differ. Its contribution
to $q_\tau$ is at least $h^2$. The lower bound in
Theorem~\ref{thm:stability} now applies.
\end{proof}

\begin{example}\label{ex:real-small-gap}
Give each edge of $H=K_{1,3}$ weight $\varepsilon$, with $0<\varepsilon<\frac{1}{6}$, and take $M=1$. A fixed-point-free involution must pair the center with one leaf and pair the other two leaves. It has two mismatched edge orbits, so $a(B_H)=2\varepsilon^2$. Since the center is distinguished by its weighted degree, no such involution is an automorphism. Nevertheless,
$$
0<\Phi(W_1(H))-T
\le\frac{5\varepsilon^2}{2(1-3\varepsilon)}
\le5\varepsilon^2.
$$
The order and the added star weight remain fixed as this excess tends to
zero. Thus arbitrary real weights do not give a uniform discrete gap.
\end{example}

The factor $N$ separating the two sides of \eqref{eq:stability} comes from
the different sizes of the star transform on and off the Fourier axes.
If the imaginary part vanishes on the axes, only the smaller star
coefficients contribute to the excess. This allows matching upper and lower
bounds as the added weight tends to infinity.

\begin{samepage}
\begin{theorem}\label{thm:regular}
Suppose all vertices of $H$ have the same weighted degree. Under the
assumptions of Theorem~\ref{thm:cone},
\begin{equation}\label{eq:regular-stability}
\frac{Na(B_H)}{4(M+s)}
\le\Phi(W_M(H))-T
\le\frac{Na(B_H)}{4(M-s)}.
\end{equation}
For a fixed such source $H$,
\begin{equation}\label{eq:regular-limit}
\lim_{M\to\infty}M\bigl(\Phi(W_M(H))-T\bigr)
=\frac{N}{4}a(B_H).
\end{equation}
\end{theorem}
\end{samepage}
\begin{proof}[\normalfont Proof]
Let the common weighted degree be $k$. In a centered labeling, the degree
vector of the extended source is zero at $0$ and equals $k$ elsewhere.
Consequently, for $v\ne0$,
$$
D_{0v}=\frac{k}{N}\sum_{y\ne0}e^{-\frac{2\pi ivy}{N}}
=-\frac{k}{N}.
$$
Symmetry gives the same assertion on the other axis, and $D_{00}$ is real
as well. Thus all the imaginary energy $q_\tau(B_H)$ lies off the axes.
At such an entry, \eqref{eq:star-transform} and $|D_{rv}|\le\frac{2s}{N}$
give
$$
|C_{rv}|\le\frac{2(M+s)}{N}.
$$
Together with the sign margin, this yields
$$
\frac{4(M-s)}{N}
\le |C_{rv}|+b_{rv}
\le\frac{4(M+s)}{N}.
$$
Use these bounds in \eqref{eq:excess-identity}, sum, and minimize exactly
as in the preceding theorem. This proves \eqref{eq:regular-stability}.
Multiplication by $M$ and the squeeze theorem prove
\eqref{eq:regular-limit}.
\end{proof}

For a simple source we can obtain a stronger gap even when the degrees are
unequal. The reason is that a small number of mismatched edge orbits cannot
place all the imaginary energy on the axes. The degree vector measures
exactly how much energy those axes can carry.

\begin{samepage}
\begin{lemma}\label{lem:axis-energy}
For a centered labeling of a weighted source, let $d$ be its weighted
degree vector on the source vertices, and let $\tau$ be the induced
pairing. Then
\begin{equation}\label{eq:axis-energy}
\sum_{v\in\mathbb{Z}_N}(\operatorname{Im}D_{0v})^2
+\sum_{r\in\mathbb{Z}_N}(\operatorname{Im}D_{r0})^2
=\frac{\|d-P_\tau d\|_2^2}{2N}.
\end{equation}
If the source is simple and $q=q_\tau(B_H)$, then
\begin{equation}\label{eq:degree-defect}
\|d-P_\tau d\|_2^2\le4q^2.
\end{equation}
\end{lemma}
\end{samepage}
\begin{proof}[\normalfont Proof]
Extend $d$ by zero at the cone vertex, and call the resulting vector
$\widetilde d$. Summing first over the row index of $B$ gives
$$
D_{0v}=\frac{1}{N}
\sum_y\widetilde d_y e^{-\frac{2\pi ivy}{N}}
=\frac{1}{\sqrt{N}}(F_N\widetilde d)_v.
$$
For a real vector, conjugation of its Fourier transform replaces the
vector by its reflection. Therefore unitarity gives
$$
\sum_v(\operatorname{Im}D_{0v})^2
=\frac{1}{4N}\|\widetilde d-R\widetilde d\|_2^2
=\frac{1}{4N}\|d-P_\tau d\|_2^2.
$$
The other axis has the same contribution. Their intersection has imaginary
part zero, so adding proves \eqref{eq:axis-energy}.

For a simple source, Proposition~\ref{prop:projection} says that $q$ is
the number of mismatched two-element edge orbits. The two unordered pairs in such an orbit have four distinct endpoints. To check this, write the orbit as $\{u,v\}$ and $\{\tau(u),\tau(v)\}$. An overlap cannot come from a fixed vertex, because $\tau$ is fixed-point-free. It must therefore give $u=\tau(v)$ or $v=\tau(u)$; applying $\tau$ gives the other equality, so the two unordered pairs would coincide. A mismatched orbit contributes two $+1$ entries
and two $-1$ entries to $d-P_\tau d$. Summing the $q$ orbit contributions,
and allowing cancellation, gives
$$
\|d-P_\tau d\|_1\le4q,
\|d-P_\tau d\|_\infty\le q.
$$
The inequality $\|v\|_2^2\le\|v\|_1\|v\|_\infty$ completes the proof.
\end{proof}

\begin{samepage}
\begin{theorem}\label{thm:simple-gap}
Let $H$ be a simple graph of even order $n\ge4$ with $m$ edges and no
fixed-point-free involutory automorphism. If $M>m$ and $N=n+1$, then
\begin{equation}\label{eq:strong-gap}
\Phi(W_M(H))\ge T+\frac{N}{16M},
\operatorname{FR}_{\min}(W_M(H))\ge\frac{T}{Q}+\frac{N}{16MQ},
\end{equation}
where $T=\frac{4Mn^2-4m}{N}$ and $Q=\sqrt{2nM^2+2m}$.
\end{theorem}
\end{samepage}
\begin{proof}[\normalfont Proof]
Fix any centered labeling and put $q=q_\tau(B_H)$. The absence of the
required automorphism implies $q\ge1$. By the reflection identity and
Lemma~\ref{lem:axis-energy}, the imaginary energy away from the axes obeys
$$
\sum_{r,v\ne0}(\operatorname{Im}D_{rv})^2
=q-\frac{\|d-P_\tau d\|_2^2}{2N}
\ge q-\frac{2q^2}{N}.
$$
Suppose first that $q\le\frac{N}{4}$. The last quantity is at least
$\frac{q}{2}\ge\frac{1}{2}$. Off the axes,
$$
|C_{rv}|\le\frac{2(M+m)}{N}<\frac{4M}{N},
$$
and hence $|C_{rv}|+b_{rv}<\frac{8M}{N}$. The contribution of these
entries to \eqref{eq:excess-identity} is therefore at least
$$
\frac{N}{8M}\sum_{r,v\ne0}(\operatorname{Im}D_{rv})^2
\ge\frac{N}{16M}.
$$
Every remaining contribution is nonnegative. If instead
$q>\frac{N}{4}$, the lower bound in \eqref{eq:label-stability} gives
$\|C\|_1-T\ge\frac{q}{4M}>\frac{N}{16M}$. Thus every centered labeling
has the claimed excess. Translation invariance covers all labelings, and
division by $Q$ gives the second assertion.
\end{proof}

\section{Computational complexity}

We now turn the separation between equality and strict inequality into a computational obstruction. The point is to make that separation large enough to detect with polynomially many binary digits. We use the following NP-complete problem: given a simple graph $H$, decide whether it has a
fixed-point-free involutory automorphism. This result is due to Lubiw
\cite{Lubiw}; the precise involutory statement is also recorded as
Theorem 1.2 in \cite{Fiala}. We may restrict to even orders $n\ge4$.
Odd orders are automatically no-instances, and the smaller cases can be
decided directly. Such cases may be sent to $2K_2$ as a fixed yes-instance or to $K_{1,3}$ as a fixed no-instance. Both have order four, so this restriction preserves NP-hardness.

We use the usual binary model of computation. A rational weight is given
by an integer numerator and a positive integer denominator written in
binary. Additive approximation with error $\delta$ means returning a
rational value whose distance from the requested real value is at most
$\delta$. A decision problem belongs to NP when a yes-answer has a certificate that can be checked in polynomial time; it is NP-complete when every problem in NP reduces to it in polynomial time. For the numerical problem below, NP-hardness means that a polynomial-time approximation algorithm would decide the stated NP-complete problem in polynomial time. The reduction makes this implication explicit.

\begin{samepage}
\begin{theorem}\label{thm:hardness}
It is NP-hard to approximate $\operatorname{FR}_{\min}(W)$ to additive
error
$$
\delta_N=\frac{1}{256N^{\frac{7}{2}}}
$$
for a nonzero nonnegative integer weighted adjacency matrix $W$ of order
$N$, with its entries given in binary. This remains true when its support
graph is connected, $N$ is odd, and all positive weights belong to
$\{1,M\}$ for one integer $M\le N^2$. The same obstruction applies to exact computation whenever the exact output permits this additive approximation to be obtained in polynomial time.
\end{theorem}
\end{samepage}
\begin{proof}[\normalfont Proof]
Let $H$ be a source graph of even order $n\ge4$ with $m$ edges. Set
$$
N=n+1,
  M=m+1,
  W=W_M(H).
$$
This construction has polynomial size. Its support is connected because
of the new universal vertex, and its positive weights belong to $\{1,M\}$.
Since $m\le\frac{n(n-1)}{2}$, we have $M\le N^2$.
Compute
$$
T=\frac{4Mn^2-4m}{N},
  Q=\sqrt{2nM^2+2m},
 \alpha=\frac{T}{Q}.
$$
Theorem~\ref{thm:cone} gives
$\operatorname{FR}_{\min}(W)=\alpha$ in a yes-instance.
Theorem~\ref{thm:simple-gap} gives
$\operatorname{FR}_{\min}(W)\ge\alpha+g$ in a no-instance, where
$$
g=\frac{N}{16MQ}.
$$
Because $m\le M^2$,
$$
Q^2=2nM^2+2m\le2NM^2.
$$
It follows that
\begin{equation}\label{eq:poly-gap}
g\ge\frac{\sqrt{N}}{16\sqrt{2} M^2}
\ge\frac{1}{16\sqrt{2} N^{\frac{7}{2}}}
>8\delta_N.
\end{equation}

Suppose an approximation algorithm returns $v$ with
$|v-\operatorname{FR}_{\min}(W)|\le\delta_N$. By rational square-root
approximation, we can compute rational numbers $\alpha_0,\beta_0$ in
polynomial time such that
$$
|\alpha_0-\alpha|\le\delta_N,
  3\delta_N\le\beta_0\le5\delta_N.
$$
For the second quantity, approximate $4\delta_N$ to error at most
$\delta_N$. Only $O(\log N)$ digits of precision are needed, in addition
to the polynomial input arithmetic. In a yes-instance,
$v-\alpha_0\le2\delta_N<\beta_0$. In a no-instance,
$$
v-\alpha_0\ge g-2\delta_N>6\delta_N>\beta_0.
$$
Comparing two rational numbers therefore decides the source problem in
polynomial time. This proves the claimed hardness.
\end{proof}

The approximation statement avoids any need to compare sums of algebraic
numbers exactly. It establishes NP-hardness; it does not assert that an
unrestricted exact-threshold problem for the Fourier ratio belongs to NP.
The conclusion about exact computation applies to an exact-output model
from which the stated approximation can be obtained in polynomial time.

\begin{samepage}
\begin{corollary}\label{cor:rational}
The same additive approximation problem is NP-hard for nonnegative
rational weights in $[0,1]$. It remains NP-hard when the positive weights
belong to $\{\frac{1}{M},1\}$ with $M\le N^2$.
\end{corollary}
\end{samepage}
\begin{proof}[\normalfont Proof]
Replace each constructed matrix $W$ by $\frac{1}{M}W$. Its Fourier ratio
is unchanged by Proposition~\ref{prop:basic}, and the rational input size
remains polynomial.
\end{proof}

For completeness, we record why a specified labeling can be evaluated to
the precision used in these reductions. The difficulty lies in selecting
the labeling, rather than in evaluating one already supplied.

\begin{samepage}
\begin{lemma}\label{lem:evaluation}
For a nonzero rational weighted adjacency matrix and a specified labeling,
its Fourier ratio can be approximated to additive error $\varepsilon>0$
in time polynomial in the input length and $\log(2+\varepsilon^{-1})$.
\end{lemma}
\end{samepage}
\begin{proof}[\normalfont Proof]
First compute the rational number $\|W\|_2^2$ and approximate its positive
square root. This allows the entries of $A=\frac{W}{\|W\|_2}$ to be
approximated to any prescribed absolute precision in polynomial time.
Their absolute values are at most one. A nonzero rational entry has size
bounded below by an exponential in the negative input length, so this
normalization requires only polynomially many extra digits.

The Fourier ratio equals $\|F_NAF_N\|_1$. Approximate the normalized
entries to absolute error at most $\eta$ and the roots of unity in
\eqref{eq:fourier} to complex absolute error at most $\eta$.
A rational approximation to $\pi$, followed by finite Taylor series for
sine and cosine on a bounded interval, provides the roots in time polynomial in $\log(2+\eta^{-1})$
and $\log N$. Each Fourier coefficient is a sum of $N^2$ terms divided by
$N$. If $|a|\le1$ and $|\zeta|=1$, approximations $a_0,\zeta_0$ with these errors satisfy
$$
|a\zeta-a_0\zeta_0|
\le |a-a_0|+|a_0||\zeta-\zeta_0|
\le 2\eta+\eta^2\le3\eta
$$
when $\eta\le1$. The resulting error in one Fourier coefficient is therefore at most $3N\eta$. Taking an absolute value does not increase an additive
complex error. Summing over $N^2$ frequencies gives error at most
$3N^3\eta$, apart from the final square-root evaluations of the moduli.
Approximate each of those moduli to error at most
$\frac{\varepsilon}{2N^2}$, and choose
$\eta\le\min\{1,\frac{\varepsilon}{6N^3}\}$. The total error is at most
$\varepsilon$. All sums, products, and square-root approximations have
polynomial bit complexity at this precision.
\end{proof}

\begin{samepage}
\begin{corollary}\label{cor:fptas}
Unless $\mathrm{P}=\mathrm{NP}$, the weighted labeling minimization
problem has no fully polynomial-time approximation scheme.
\end{corollary}
\end{samepage}
\begin{proof}[\normalfont Proof]
Such a scheme would return a labeling of value at most
$(1+\varepsilon)\operatorname{FR}_{\min}(W)$ in time polynomial in the
input length and $\varepsilon^{-1}$. Take
$\varepsilon=\frac{1}{1024N^5}$. Since
$\operatorname{FR}_{\min}(W)\le N$, the returned value would exceed the
minimum by at most $\frac{1}{1024N^4}$. Use
Lemma~\ref{lem:evaluation} to evaluate it with the same additive error.
The resulting rational number would approximate the minimum to error at
most $\frac{1}{512N^4}$, which is smaller than $\delta_N$. The running
time would still be polynomial, contradicting Theorem~\ref{thm:hardness}
unless $\mathrm{P}=\mathrm{NP}$.
\end{proof}

\begin{remark}\label{rem:multigraph}
An integer-weighted graph can be viewed as a loopless multigraph by
replacing an edge of weight $M$ with $M$ parallel edges. In our reduction,
the total multiplicity is $nM+m=O(n^3)$. Thus the hardness statement also
holds for explicitly represented multigraphs, with edge complexity defined
from the multiplicity matrix. Passing to a simple graph requires a further
argument, because replacing parallel edges by a gadget changes the vertex
set and its possible labelings.
\end{remark}

\section{The unweighted problem}

What remains if all positive weights must equal one? The reflection identity still applies, so we can ask how many edge changes are needed before a real Fourier labeling becomes possible. This question has an exact combinatorial answer. The difficulty is then to understand its relation to minimizing the full Fourier norm. To see what survives, let $A$ be the adjacency
matrix of a simple graph $G$ of order $N$ and define
\begin{equation}\label{eq:eta}
\eta(G)=\min_\pi\|\operatorname{Im}(F_NA_\pi F_N)\|_2^2.
\end{equation}
The empty graph is allowed here, since there is no denominator.

\begin{samepage}
\begin{proposition}\label{prop:eta-edits}
For $N\ge3$, let $d=\gcd(2,N)$. Then $\eta(G)$ is the minimum number of
edge additions or deletions required to obtain, on the same vertex set, a
simple graph having an involutory automorphism with exactly $d$ fixed
vertices. In particular, $\eta(G)$ is a nonnegative integer.
\end{proposition}
\end{samepage}
\begin{proof}[\normalfont Proof]
As the labeling varies, reflection induces all involutions with exactly
$d$ fixed vertices. Each such involution is conjugate to reflection by a
labeling, since its fixed points and transposed pairs have the same sizes.
For a fixed involution, the reflection identity gives one fourth of the
squared adjacency difference. The edge-orbit calculation in
Proposition~\ref{prop:projection} identifies this with the number of
mismatched two-element orbits. That calculation also applies when the
involution has fixed vertices: an orbit of an unordered pair still has
one or two elements. One edit repairs each mismatched orbit. Taking the
minimum proves the assertion.
\end{proof}

\begin{samepage}
\begin{theorem}\label{thm:reality-hard}
Deciding whether $\eta(G)=0$ is NP-complete, even for connected simple
graphs of odd order with a unique universal vertex.
\end{theorem}
\end{samepage}
\begin{proof}[\normalfont Proof]
Let $H$ be a source graph of even order $n\ge4$. Add two isolated vertices
to $H$, then add a vertex $c$ adjacent to every vertex of this disjoint
union. In graph notation, the result is
$$
G=K_1\vee(H\sqcup2K_1),
$$
where $\sqcup$ denotes disjoint union and $\vee$ denotes the join. Its
order is $n+3$, which is odd. The vertex $c$ has degree $n+2$ and is the
unique universal vertex. The old vertices have degree at most $n$, and
the two additional vertices have degree one.

If $H$ has a fixed-point-free involutory automorphism, extend it by
exchanging the two added isolated vertices and fixing $c$. Assign $c$
label zero and every transposed pair opposite labels. The reflection
identity gives $\eta(G)=0$.

Conversely, a real Fourier labeling induces an involutory automorphism
with exactly one fixed vertex. Every automorphism fixes $c$, since it is
the unique universal vertex. The induced action on $H\sqcup2K_1$ is
therefore fixed-point-free. It preserves the nonisolated vertices of $H$,
whose degrees within the disjoint union are positive. If $H$ has $k$
isolated vertices, then all $k+2$ isolated vertices in that union are
paired, so $k$ is even. Retain the induced action on the nonisolated
vertices and pair the original isolated vertices among themselves. This
is a fixed-point-free involutory automorphism of $H$.

The construction proves NP-hardness. Membership in NP does not require
numerical Fourier calculations. A labeling is a certificate, and the
identities $A_\pi(x,y)=A_\pi(-x,-y)$ can be checked in $O(N^2)$ operations.
\end{proof}

It is tempting to use this recognition theorem to obtain hardness of the
full Fourier norm on simple graphs. The next example explains the first
obstruction. Reality alone does not prescribe the signs needed in the
cone proof.

\begin{example}\label{ex:sign-failure}
Take the path $H=P_4$, assign its successive vertices labels $1,2,3,4$
in $\mathbb{Z}_5$, and add a universal vertex of label $0$ with weight
$M=1$. This labeling is invariant under reflection, so its Fourier matrix
is real. Direct substitution in \eqref{eq:fourier} gives
$$
(F_5WF_5)_{22}=\frac{\sqrt{5}-1}{5}>0.
$$
The sign prescribed by the cone argument at this entry is negative.
Summing the absolute values gives
$$
\|F_5WF_5\|_1=\frac{48+4\sqrt{5}}{5}>\frac{52}{5}.
$$
Here $\frac{52}{5}$ is the expression $T$ obtained by inserting $M=1$ and $m=3$. The signed-sum identity remains valid, but the assumption $M>m$ is absent. Some real coefficients now have the wrong signs, so reality no longer gives equality.
\end{example}

There is a second obstruction. Even when a real Fourier labeling exists,
minimizing the full norm may require leaving that class of labelings.

\begin{samepage}
\begin{proposition}\label{prop:counterexample}
There is a connected simple graph with a real Fourier labeling for which
every minimizing labeling has a nonreal Fourier matrix.
\end{proposition}
\end{samepage}
\begin{proof}[\normalfont Proof]
Let $G$ consist of a universal vertex $c$, two disjoint triangles, and one
additional edge joining a vertex of the first triangle to a vertex of the
second. It has seven vertices and thirteen edges. Exchanging the two
triangles appropriately gives an involution fixing only $c$, so real
Fourier labelings exist.

We first describe all real labelings up to changes that preserve the norm.
The unique universal vertex must have label $0$, since it is fixed by
every automorphism and reflection in $\mathbb{Z}_7$ has only one fixed
label. The two endpoints of the joining edge are exactly the vertices of
degree four, so they must receive opposite labels. Multiplication by a
nonzero residue lets us give them labels $1$ and $6$. Reflection then
exchanges the two triangles. The triangle containing $1$ must contain
one member of $\{2,5\}$ and one member of $\{3,4\}$. These choices give
exactly four labeled edge sets. They are listed in
Table~\ref{tab:real-types}; the other triangle is always the negative of
the listed one, and the joining edge is $\{1,6\}$.

\begin{table}[htbp]
\centering
\begin{tabular}{cc}
\toprule
Triangle containing $1$ & Rational interval for $\|F_7A F_7\|_1$\\
\midrule
$\{1,2,3\}$ & $(22.481383,22.481384)$\\
$\{1,2,4\}$ & $(25.714285,25.714286)$\\
$\{1,3,5\}$ & $(22.249490,22.249491)$\\
$\{1,4,5\}$ & $(22.057990,22.057991)$\\
\bottomrule
\end{tabular}
\caption{The four real labeling types after normalizing the joining edge. The terminating decimal endpoints denote exact rational numbers.}
\label{tab:real-types}
\end{table}

Now keep $c$ at $0$, put the triangles at $\{1,2,3\}$ and $\{4,5,6\}$,
and use joining edge $\{1,4\}$. Its Fourier matrix is nonreal: reflection
would send its joining edge to the absent edge $\{3,6\}$. Its norm lies
in the rational interval
$$
(21.953290,21.953291),
$$
which is below $22$. Appendix~\ref{app:certificate} proves all five
interval inclusions using rational arithmetic, with the full certificate
included there. Thus this single nonreal labeling has smaller norm than
every real labeling. All Fourier ratios have the same denominator
$\sqrt{26}$, so the same strict comparison holds for the ratios.
Figure~\ref{fig:seven-vertex} displays the comparison.
\end{proof}

\begin{figure}[htbp]
\centering
\begin{tikzpicture}[scale=0.77,
 vertex/.style={circle,draw,fill=white,minimum size=6.5mm,inner sep=1pt},
 graph edge/.style={line width=0.8pt},
 universal edge/.style={draw=black!35,line width=0.55pt}]
\begin{scope}
\node at (0,3.15) {Real Fourier labeling};
\coordinate (c) at (0,2.25);
\coordinate (a) at (-0.55,0);
\coordinate (b) at (0.55,0);
\coordinate (u) at (-2.3,1.1);
\coordinate (v) at (-2.3,-1.1);
\coordinate (w) at (2.3,1.1);
\coordinate (z) at (2.3,-1.1);
\draw[universal edge] (c)--(a) (c)--(b) (c)--(u)
 (c)--(v) (c)--(w) (c)--(z);
\draw[graph edge] (a)--(u)--(v)--(a)--(b)--(w)--(z)--(b);
\node[vertex] at (c) {$0$};
\node[vertex] at (a) {$1$};
\node[vertex] at (b) {$6$};
\node[vertex] at (u) {$4$};
\node[vertex] at (v) {$5$};
\node[vertex] at (w) {$3$};
\node[vertex] at (z) {$2$};
\node at (0,-1.9) {$\|F_7AF_7\|_1>22$};
\end{scope}
\begin{scope}[xshift=7.2cm]
\node at (0,3.15) {Nonreal Fourier labeling};
\coordinate (c) at (0,2.25);
\coordinate (a) at (-0.55,0);
\coordinate (b) at (0.55,0);
\coordinate (u) at (-2.3,1.1);
\coordinate (v) at (-2.3,-1.1);
\coordinate (w) at (2.3,1.1);
\coordinate (z) at (2.3,-1.1);
\draw[universal edge] (c)--(a) (c)--(b) (c)--(u)
 (c)--(v) (c)--(w) (c)--(z);
\draw[graph edge] (a)--(u)--(v)--(a)--(b)--(w)--(z)--(b);
\node[vertex] at (c) {$0$};
\node[vertex] at (a) {$1$};
\node[vertex] at (b) {$4$};
\node[vertex] at (u) {$2$};
\node[vertex] at (v) {$3$};
\node[vertex] at (w) {$5$};
\node[vertex] at (z) {$6$};
\node at (0,-1.9) {$\|F_7AF_7\|_1<22$};
\end{scope}
\end{tikzpicture}
\caption{Two labelings of the same seven-vertex graph. The gray edges
join the universal vertex to every other vertex. The left labeling is
invariant under negation modulo seven and has the smallest norm among
the four real types. The right labeling has smaller norm but lacks that
symmetry. Both Fourier ratios have denominator $\sqrt{26}$.}
\label{fig:seven-vertex}
\end{figure}
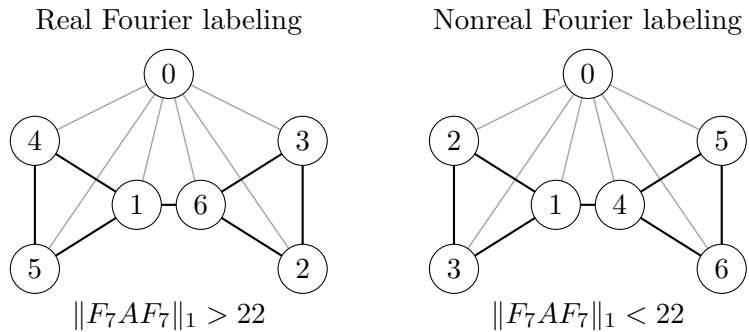

The classification also explains the count of real labelings. If the
vertices of the abstract graph are distinguished, one chosen endpoint of
the joining edge has six possible nonzero labels, and its partner is then
forced. There are four choices of the triangle containing it, followed by
two assignments of the remaining vertices in each triangle. Thus there
are $6\cdot4\cdot2\cdot2=96$ centered real labelings. The proof only needs
the four resulting norm types and one competing labeling; it does not
require the unrestricted minimum to be computed.

\FloatBarrier

A different possible approach is to replace large weights by extra vertices. Before attempting this, it is useful to understand an ordinary uniform blow-up. The next proposition shows what happens for one natural labeling. It also explains why this operation alone does not give an unweighted hardness reduction.

Let $W^{(k)}$ be the uniform independent-set blow-up of a nonzero weighted
adjacency matrix $W$. Each vertex is replaced by $k$ independent vertices,
and each pair of vertex classes receives all the edges of the original
weight. The weights themselves are retained in this operation.

\begin{samepage}
\begin{proposition}\label{prop:blowup}
For every positive integer $k$,
$$
\operatorname{FR}_{\min}(W^{(k)})
\le\operatorname{FR}_{\min}(W).
$$
If $\operatorname{FR}_{\min}(W)=\frac{\mathcal E(W)}{\|W\|_2}$, then
equality holds.
\end{proposition}
\end{samepage}
\begin{proof}[\normalfont Proof]
Choose a minimizing labeling of $W$ on $\mathbb{Z}_N$. On $\mathbb{Z}_{kN}$,
label its blow-up by
$$
W^{(k)}(x,y)=W(x\bmod N,y\bmod N).
$$
Write $x=u+aN$ and $y=w+bN$, with $u,w\in\mathbb{Z}_N$ and
$0\le a,b<k$. The two sums over $a,b$ are geometric sums. Each vanishes
unless the corresponding frequency is divisible by $k$, in which case it
equals $k$. With the normalization $\frac{1}{kN}$, this gives
$$
\widehat{W^{(k)}}(r,v)=
\begin{cases}
k\widehat W(\frac{r}{k},\frac{v}{k}),&k\mid r\text{ and }k\mid v,\\
0,&\text{otherwise}.
\end{cases}
$$
Thus both Fourier norms are multiplied by $k$, proving the upper bound.

As an abstract matrix, $W^{(k)}$ is permutation-similar to $W\otimes J_k$,
where $J_k$ is the all-ones matrix. The eigenvalues of $J_k$ are $k$ and
$k-1$ zeros. Hence
$$
\mathcal E(W^{(k)})=k\mathcal E(W),
  \|W^{(k)}\|_2=k\|W\|_2.
$$
When the original graph attains the energy bound, applying that bound to
the blow-up gives the reverse inequality.
\end{proof}

A hardness reduction needs a lower bound valid for every labeling of the constructed graph. The blow-up proposition supplies one useful labeling, but by itself it does not control all the others. A gadget replacing an edge of weight $M$ changes the cyclic group and introduces additional vertex permutations. Controlling those permutations is the remaining difficulty.

\begin{samepage}
\begin{question}\label{q:simple}
Is computing $\operatorname{FR}_{\min}(G)$ NP-hard for simple unweighted
graphs? Is additive approximation to some inverse-polynomial accuracy
already NP-hard on this class?
\end{question}
\end{samepage}

\section{Further questions}

The proof has used three elementary facts in succession. A large star fixes the Fourier signs. Conjugation identifies the imaginary part with a reflection defect. Finally, the discrete edge weights keep a nonzero defect large enough for a polynomial-time computation to detect. These connections suggest several further questions.

The constant-degree theorem suggests a structural question. On which graph classes can the defect $a(B_H)$ be computed efficiently? Even when the general problem is difficult, special classes may allow the pairing to be found from their decomposition or degree structure. The explicit connection with the leading Fourier excess may help transfer information between these two problems.

For unequal degrees, Lemma~\ref{lem:axis-energy} identifies the exact amount of imaginary energy on the axes. A more detailed stability theory could retain both $q_\tau(B_H)$ and $\|d-P_\tau d\|_2^2$ instead of bounding the latter crudely. The distinction is useful because the Fourier entries on and off the axes have different sizes even when $M$ is large.

The simple-graph problem asks for a further idea. The seven-vertex example shows that imposing reality can remove every minimizer, while the blow-up calculation leaves an uncontrolled minimization over new labels. A useful unweighted construction would have to address these features directly. It is also natural to ask whether the edit parameter $\eta(G)$ can give
bounds for the full edge complexity under additional graph hypotheses.

Finally, the reflection identity applies to every real matrix, and the
cone argument extends to signed source weights. Let
$$
s=\sum_{u<v}(B_H)_{uv},
  L=\sum_{u<v}|(B_H)_{uv}|,
$$
and assume $M>L$. Then $|D_{rv}|\le\frac{2L}{N}$ gives the same sign pattern. The signed sum still equals $\frac{4Mn^2-4s}{N}$, so the equality
characterization is unchanged. The stability lower bound remains
$\frac{a(B_H)}{4M}$, and its upper bound becomes $\frac{Na(B_H)}{4(M-L)}$. In this setting, averaging gives the closest
invariant signed symmetric zero-diagonal matrix. Nonnegative weights were used throughout the main discussion to keep the graph interpretation
direct and to include the integer instances needed for hardness.

\appendix
\section{A rational certificate for the seven-vertex example}
\label{app:certificate}

This appendix gives all the arithmetic needed for the five interval
inclusions in Proposition~\ref{prop:counterexample}. The proof of the
classification of real labelings is in that proposition. The computation
below only evaluates the four representatives and the one competing labeling. It uses rational numbers and integer square roots throughout.

Put $R=10^{20}$ and $\epsilon=10^{-18}$. We first obtain an approximation to each seventh root of unity whose real and imaginary coordinates have error less than $\epsilon$. Machin's identity is
$$
\pi=16\arctan\frac{1}{5}-4\arctan\frac{1}{239}.
$$
The identity follows, for example, from the tangent addition formula:
if $a=\arctan\frac{1}{5}$, then $\tan(2a)=\frac{5}{12}$ and
$\tan(4a)=\frac{120}{119}$, so
$\tan(4a-\arctan\frac{1}{239})=1$. The angle lies in
$(0,\frac{\pi}{2})$, which determines it as $\frac{\pi}{4}$.

The first thirty terms of the alternating series for each arctangent,
together with the next term, enclose its value between two rational
numbers. They give rational bounds $p_-\le\pi\le p_+$. The certificate
checks that, for $p_0=\frac{\lfloor Rp_-\rfloor}{R}$,
$$
0\le\pi-p_0\le\frac{2}{R}.
$$
For $-3\le k\le3$, set $x_k=-\frac{2p_0k}{7}$. Then $|x_k|<3$ and its
distance from $-\frac{2\pi k}{7}$ is at most $\frac{12}{7R}$. Taylor
polynomials for cosine through degree $40$ and sine through degree $41$
have errors at most
$$
\frac{|x_k|^{41}}{41!}
 \text{ and } 
\frac{|x_k|^{42}}{42!},
$$
respectively. Sine and cosine are both Lipschitz with constant one.
Rounding each polynomial value down to a multiple of $\frac{1}{R}$
therefore gives coordinate errors bounded by the Taylor remainder plus
$\frac{12}{7R}+\frac{1}{R}$. The code verifies by rational comparisons
that both bounds are smaller than $\epsilon$.

Every graph being evaluated has twenty-six ordered edges. Replacing the roots in one Fourier coefficient by these rational approximations changes each coordinate by at most $\frac{26\epsilon}{7}$. The corresponding modulus changes by at most $\frac{52\epsilon}{7}$, using the sum of the two coordinate errors as a bound for their Euclidean norm. Summing over forty-nine entries gives a total root-approximation error at most $364\epsilon$.

An approximate coefficient has the form $\frac{a+ib}{7R}$ for integers $a,b$. If $h=\lfloor\sqrt{a^2+b^2}\rfloor$, its modulus lies between $\frac{h}{7R}$ and $\frac{h+1}{7R}$. Thus, integer square roots also give rational bounds for the sum of the approximate moduli. Enlarging them by $364\epsilon$ encloses the true Fourier norm. The program below checks that this entire enclosure lies strictly inside each printed interval. No floating-point tolerance is used.

The code is complete and requires only Python 3.8 or later and its standard library. Saving the following text as a Python file and running it prints the five certified intervals and a final success message.

\begingroup
\small
\begin{verbatim}
from fractions import Fraction as Q
from itertools import combinations
from math import factorial, isqrt

def atan_bounds(b):
    x = Q(1, b)
    total = sum(((-1)**j * x**(2*j + 1) / (2*j + 1)
                 for j in range(30)), Q(0))
    following = x**61 / 61
    return total, total + following

a, b = atan_bounds(5)
c, d = atan_bounds(239)
pi_lower, pi_upper = 16*a - 4*d, 16*b - 4*c
R = 10**20
eps = Q(1, 10**18)
p0 = Q((pi_lower * R).__floor__(), R)
assert p0 <= pi_lower <= pi_upper <= p0 + Q(2, R)
roots = []
for j in range(7):
    k = j if j <= 3 else j - 7
    x = -2*p0*k/7
    assert abs(x) < 3
    cosine = sum(((-1)**h * x**(2*h) / factorial(2*h)
                  for h in range(21)), Q(0))
    sine = sum(((-1)**h * x**(2*h + 1)
                / factorial(2*h + 1)
                for h in range(21)), Q(0))
    common_error = Q(12, 7*R) + Q(1, R)
    assert abs(x)**41 / factorial(41) + common_error < eps
    assert abs(x)**42 / factorial(42) + common_error < eps
    roots.append(((cosine * R).__floor__(),
                  (sine * R).__floor__()))

def graph(first, second, bridge):
    edges = [(0, j) for j in range(1, 7)]
    edges += list(combinations(first, 2))
    edges += list(combinations(second, 2))
    edges += [bridge]
    edges = {tuple(sorted(e)) for e in edges}
    assert len(edges) == 13
    return edges

def norm_interval(edges):
    ordered = list(edges) + [(v, u) for u, v in edges]
    lower_sum = 0
    for r in range(7):
        for s in range(7):
            counts = [0] * 7
            for u, v in ordered:
                counts[(r*u + s*v) % 7] += 1
            a = sum(counts[j]*roots[j][0] for j in range(7))
            b = sum(counts[j]*roots[j][1] for j in range(7))
            lower_sum += isqrt(a*a + b*b)
    error = 364*eps
    return (Q(lower_sum, 7*R) - error,
            Q(lower_sum + 49, 7*R) + error)

def reflected(edges):
    return {tuple(sorted(((-u) % 7, (-v) % 7)))
            for u, v in edges}

def check(edges, left, right, real):
    assert (edges == reflected(edges)) == real
    lower, upper = norm_interval(edges)
    assert Q(left) < lower <= upper < Q(right)
    if real:
        assert Q(left) > 22
    else:
        assert Q(right) < 22
    print(left, '< norm <', right)

cases = [
    ((1, 2, 3), '22.481383', '22.481384'),
    ((1, 2, 4), '25.714285', '25.714286'),
    ((1, 3, 5), '22.249490', '22.249491'),
    ((1, 4, 5), '22.057990', '22.057991'),
]
for first, left, right in cases:
    second = tuple((-v) % 7 for v in first)
    check(graph(first, second, (1, 6)), left, right, True)
check(graph((1, 2, 3), (4, 5, 6), (1, 4)),
      '21.953290', '21.953291', False)
print('All five rational interval certificates verified.')
\end{verbatim}
\endgroup

\end{document}